\documentclass[12pt]{article}
\usepackage{amssymb}
\usepackage{amsmath}
\usepackage{amsthm}
\usepackage{color}
\usepackage{comment}
\usepackage{geometry}		
\usepackage{graphicx}

\let\originalleft\left
\let\originalright\right
\renewcommand{\left}{\mathopen{}\mathclose\bgroup\originalleft}
\renewcommand{\right}{\aftergroup\egroup\originalright}
\newcommand{\myStep}[2]{{\bf Step #1} --- #2\\}

\begin{document}

\newcommand{\bn}{{\bf n}}
\newcommand{\bO}{{\bf 0}}
\newcommand{\bv}{{\bf v}}
\newcommand\cO{\mathcal{O}}
\newcommand{\ee}{\varepsilon}
\newcommand\cN{\mathcal{N}}
\newcommand{\rD}{{\rm D}}

\newcommand{\removableFootnote}[1]{\footnote{#1}}

\newtheorem{theorem}{Theorem}[section]
\newtheorem{corollary}[theorem]{Corollary}
\newtheorem{lemma}[theorem]{Lemma}
\newtheorem{proposition}[theorem]{Proposition}

\theoremstyle{definition}
\newtheorem{definition}{Definition}[section]
\newtheorem{example}[definition]{Example}

\theoremstyle{remark}
\newtheorem{remark}{Remark}[section]



\title{
Unfolding globally resonant homoclinic tangencies.
}
\author{
Sishu Shankar Muni, Robert I.~McLachlan, David J.W.~Simpson\\\\
School of Fundamental Sciences\\
Massey University\\
Palmerston North\\
New Zealand
}
\maketitle


\begin{abstract}
Global resonance is a mechanism by which a homoclinic tangency of a smooth map can have infinitely many asymptotically stable, single-round periodic solutions. To understand the bifurcation structure one would expect to see near such a tangency, in this paper we study one-parameter perturbations of typical globally resonant homoclinic tangencies. We assume the tangencies are formed by the stable and unstable manifolds of saddle fixed points of two-dimensional maps. We show the perturbations display two infinite sequences of bifurcations, one saddle-node the other period-doubling, between which single-round periodic solutions are asymptotically stable. Generically these scale like $|\lambda|^{2 k}$, as $k \to \infty$, where $-1 < \lambda < 1$ is the stable eigenvalue associated with the fixed point. If the perturbation is taken tangent to the surface of codimension-one homoclinic tangencies, they instead scale like $\frac{|\lambda|^k}{k}$. We also show slower scaling laws are possible if the perturbation admits further degeneracies.
\end{abstract}

\section{Introduction}
\label{sec:intro}
Homoclinic tangencies are perhaps the simplest mechanism in nonlinear dynamical systems for the loss of hyperbolicity and the creation of chaotic dynamics \cite{PaTa93}.  They occur most simply for saddle fixed points of two-dimensional maps.  A tangential intersection between the stable and unstable manifolds of a fixed point is a homoclinic tangency, see Fig 1.  This intersection is one point of an orbit that is homoclinic to the fixed point. This is a codimension-one phenomenon, meaning it can be equated to a single scalar condition.  At a homoclinic tangency the map typically has infinitely many periodic solutions. Some of these are {\em single-round}, roughly meaning that they shadow the homoclinic orbit once before repeating, Fig.~\ref{fig:TangUnf2}. Newhouse in \cite{Ne74} showed that infinitely many stable multi-round periodic solutions can coexist at a homoclinic tangency. At a generic homoclinic tangency all single-round periodic solutions of sufficiently large period are unstable \cite{GaSi72,GaSi73}.  
\begin{figure}[htbp!]
\begin{center}
\includegraphics[width=6cm]{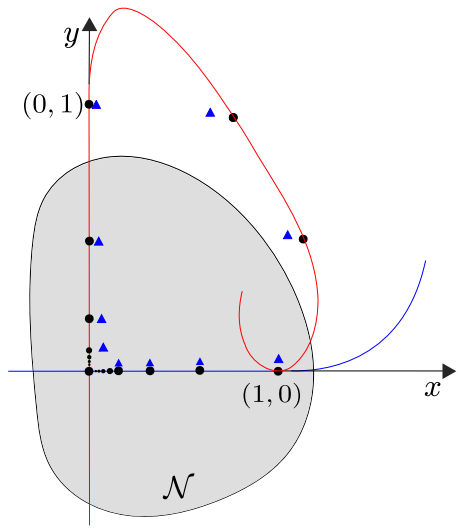}
\caption{A homoclinic tangency for a saddle fixed point of a two-dimensional map. In this illustration the eigenvalues associated with the fixed points are positive, i.e. $0 < \lambda < 1$ and $\sigma > 1$.  A coordinate change has been applied so that in the region $\cN$ (shaded) the coordinate axes coincide with the stable and unstable manifolds.  The homoclinic orbit $\Gamma_{\rm HC}$ is shown with black dots. A typical single-round periodic solution is shown with blue triangles.} 
\label{fig:TangUnf2}
\end{center}
\end{figure}
In our previous work \cite{MuMcSi21} we determined what degeneracies are needed for the infinitely many single-round periodic solutions to be asymptotically stable for smooth maps on $\mathbb{R}^2$. We found that the eigenvalues associated with the saddle fixed point, call them $\lambda$ and $\sigma$ satisfying $0 < |\lambda| < 1 < |\sigma|$, need to multiply to $1$ or $-1$.  If $\lambda \sigma = -1$ (so the map is orientation-reversing, at least locally), then the phenomenon is codimension-three.  It involves a `global resonance' condition on the reinjection mechanism for taking iterates back to a neighbourhood of the fixed point.

If a family of maps $f_\mu$, with $\mu \in \mathbb{R}^n$, has this phenomenon at some point $\mu^*$ in parameter space, then for any positive $j \in \mathbb{Z}$ there exists an open set containing $\mu^*$ in which $j$ asymptotically stable, single-round periodic solutions coexist.  Due to the high codimension, a precise description of the shape of these sets (for large $j$) is beyond the scope of this paper.  The approach we take here is to consider one-parameter families that perturb from a globally resonant homoclinic tangency.  Some information about the size and shapes of the sets can then be inferred from our results. Globally resonant homoclinic tangencies are hubs for extreme multi-stability.  They should occur generically in some families of maps with three or more parameters, such as the
 generalised H\'enon map \cite{GoKu05,KuMe19}, but to our knowledge they have yet to have been identified.

We find that as the value of $\mu$ is varied from $\mu^*$, there occurs infinite sequence of either saddle-node or period-doubling bifurcations that destroy the periodic solutions or make them unstable. Generically these sequences converge exponentially to $\mu^*$ with the distance (in parameter space) to the bifurcation asymptotically proportional to $|\lambda|^{2 k}$, where the periodic solutions have period $k + m$, for some fixed $m \ge 1$. If we move away from $\mu^*$ without a linear change to the codimension-one condition of a homoclinic tangency, the bifurcation values instead generically scale like $\frac{|\lambda|^k}{k}$. If the perturbation suffers further degeneracies, the scaling can be slower. Specifically we observe $|\lambda|^k$ and $\frac{1}{k}$ for an abstract example that we believe is representative of how the bifurcations scale in general.

Similar results have been obtained for more restrictive classes of maps. For area-preserving families the phenomenon is codimension-two and there exist infinitely many elliptic, single-round periodic solutions \cite{GoSh05}. As shown in \cite{DeGo15,GoGo09} the periodic solutions are destroyed or lose stability in bifurcations that scale like $|\alpha|^{2 k}$, matching our result. For piecewise-linear families the phenomenon is codimension-three \cite{Si14,SiTu17}.  In this setting the bifurcation values instead scale like $|\alpha|^k$ \cite{Si14b}, see also \cite{DoLa08}.

The phenomenon is also reminscient of a celebrated result of Newhouse \cite{Ne74}. Newhouse showed that for a generic (codimension-one) homoclinic tangency with $|\lambda \sigma| < 1$, infinitely many asymptotically stable periodic solutions coexist for a dense set of parameter values near that of the tangency.  However, in this scenario the periodic solutions are {\em multi-round} (shadow the homoclinic orbit several times before repeating).

The remainder of the paper is organised as follows. In \S2 we introduce two maps, $T_0$ and $T_1$, to describe the dynamics near the homoclinic orbit.  The map $T_0$ is a local map describing the saddle fixed point, while $T_1$ represents the reinjection mechanism.  The single-round periodic solutions correspond to fixed points of $T_0^k \circ T_1$. In \S3 we summarise the necessary and sufficient conditions derived in \cite{MuMcSi21} for the coexistence of infinitely many asymptotically stable, single-round periodic solutions. Then in \S4--6 we introduce parameter-dependence and state our main result (Theorem \ref{th:main}) for the scaling properties of the saddle-node and period-doubling bifurcations. A proof of Theorem \ref{th:main} is provided in \S7. Next in \S8 we illustrate Theorem 6.1 with a four-parameter family of $C^1$ maps (an abstract minimal example).  We observe numerically computed bifurcation values match those predicted by Theorem \ref{th:main}, to leading order, and observe slower scaling laws in special cases. Finally \S 9 provides a discussion and outlook for further studies.

\section{A quantitative description of the dynamics near a homoclinic connection}
\label{sec:localCoordinates}
\setcounter{equation}{0}

Let $f$ be a $C^\infty$ map on $\mathbb{R}^2$.
Suppose the origin $(x,y) = (0,0)$ is a saddle fixed point of $f$.
That is, $\rD f(0,0)$ has eigenvalues $\lambda, \sigma \in \mathbb{R}$ satisfying
\begin{equation}
0 < |\lambda| < 1 < |\sigma|.
\label{eq:eigAssumption}
\end{equation}
By the results of Sternberg \cite{ShSh98,St58}
there exists a $C^\infty$ coordinate change that transforms $f$ to
\begin{equation}
T_0(x,y) = \begin{bmatrix}
\lambda x \left( 1 + \cO \left( x y \right) \right) \\
\sigma y \left( 1 + \cO \left( x y \right) \right)
\end{bmatrix}.
\label{eq:T0general}
\end{equation}
In these new coordinates let $\cN$ be a convex neighbourhood of the origin for which
\begin{equation}
f(x,y) = T_0(x,y), \qquad \text{for all $(x,y) \in \cN$}.
\label{eq:fEqualsT0}
\end{equation}

See Fig.~\ref{fig:TangUnf2}.  If $\lambda^p \sigma^q \ne 1$ for all integers $p, q \ge 1$, then the eigenvalues are said to be {\em non-resonant} and the coordinate change can be chosen so that $T_0$ is linear.  If not, then $T_0$ must contain resonant terms that cannot be eliminated by the coordinate change. As explained in \S\ref{sec:parameters}, if $\lambda \sigma = 1$ we can reach the form
\begin{equation}
T_0(x,y) = \begin{bmatrix}
\lambda x \left( 1 + a_1 x y + \cO \left( x^2 y^2 \right) \right) \\
\sigma y \left( 1 + b_1 x y + \cO \left( x^2 y^2 \right) \right)
\end{bmatrix},
\label{eq:T0}
\end{equation}
where $a_1, b_1 \in \mathbb{R}$.
If $\lambda \sigma = -1$ we can obtain \eqref{eq:T0} with $a_1 = b_1 = 0$.

Now suppose there exists an orbit homoclinic to the origin, $\Gamma_{\rm HC}$.
By scaling $x$ and $y$ we may assume that $(1,0)$ and $(0,1)$ are points on $\Gamma_{\rm HC}$ and
\begin{equation}
(1,0), (\lambda,0), \left( 0, \tfrac{1}{\sigma} \right), \left( 0, \tfrac{1}{\sigma^2} \right) \in \cN,
\label{eq:pointsInN}
\end{equation}
By assumption there exists $m \ge 1$ such that $f^m(0,1) = (1,0)$.
We let $T_1 = f^m$ and expand $T_1$ in a Taylor series centred at $(x,y) = (0,1)$:
\begin{equation}
T_1(x,y) = \begin{bmatrix}
c_0 + c_1 x + c_2 (y-1) + \cO \left( (x,y-1)^2 \right) \\
d_0 + d_1 x + d_2 (y-1) + d_3 x^2 + d_4 x(y-1) + d_5 (y-1)^2 + \cO \left( (x,y-1)^3 \right)
\end{bmatrix},
\label{eq:T1}
\end{equation}
where $c_0 = 1$ and $d_0 = 0$. In \eqref{eq:T1} we have written explicitly the terms that will be important below.

\section{Conditions for infinitely many stable single-round solutions}
\label{sec:conditions}
\setcounter{equation}{0}

In this section we state the main results of our previous work \cite{MuMcSi21}.  First Theorem \ref{le:necessaryConditions} gives necessary conditions for the existence of infinitely many stable, single-round periodic solutions.  Then Theorem \ref{le:sufficientConditions} gives sufficient conditions for these to exist and be asymptotically stable.

\begin{theorem}
Suppose $f$ satisfies \eqref{eq:fEqualsT0} with \eqref{eq:T0general} and \eqref{eq:pointsInN}.
Suppose in \eqref{eq:T1} we have $c_0 = 1$, $d_0 = 0$, $d_5 \ne 0$, and $c_1 d_2 - c_2 d_1 \ne 0$.
Suppose $f$ has an infinite sequence of stable, single-round, periodic solutions accumulating on $\Gamma_{\rm HC}$.
Then
\begin{align}
d_2 &= 0, \label{eq:cond1} \\
|\lambda \sigma| &= 1, \label{eq:cond2} \\
|d_1| &= 1, \label{eq:cond3}
\end{align}
with $d_1 = 1$ in the case $\lambda \sigma = 1$.
Moreover, if $\lambda \sigma = 1$ and $T_0$ has the form \eqref{eq:T0} then
\begin{equation}
a_1 + b_1 = 0.
\label{eq:cond4}
\end{equation}
\label{le:necessaryConditions}
\end{theorem}

The equation $d_2 = 0$ corresponds to a homoclinic tangency, as shown in Fig.~\ref{fig:TangUnf2}.
With $|\lambda \sigma| = 1$ we have $|\det(\rD f)| = 1$ at the origin.
The condition $d_1 = 1$ is a global condition, see \cite{MuMcSi21} for a geometric interpretation,
with which the tangency is termed globally resonant.
Finally if $\lambda \sigma = 1$ and $T_0$ has the form \eqref{eq:T0}, then
$\det(\rD f) = 1 + 2 (a_1 + b_1) x y + \cO \left( x^2 y^2 \right)$.
Thus the condition $a_1 + b_1 = 0$ implies that as $(x,y)$ is varied from $(0,0)$, the value of $\det(\rD f)$ varies quadratically instead of linearly (as is generically the case).
Now suppose $|\lambda \sigma| = 1$.
Given $k_{\rm min} \in \mathbb{Z}$, let
\begin{equation}
K(k_{\rm min}) = \left\{ k \in \mathbb{Z} \,\big|\, k \ge k_{\rm min}, \left( \lambda \sigma \right)^k = d_1 \right\}.
\label{eq:K}
\end{equation}
If $\lambda \sigma = d_1 = 1$, then $K$ is the set of all integers greater than or equal to $k_{\rm min}$.
If $\lambda \sigma = -1$ and $d_1 = 1$ [resp.~$d_1 = -1$],
then $K$ is the set of all even [resp.~odd] integers greater than or equal to $k_{\rm min}$.
Also let
\begin{equation}
\Delta = \left( 1 - c_2 - d_1 d_4 \right)^2 - 4 d_5 (d_3 + c_1 d_1).
\label{eq:Delta}
\end{equation}

\begin{theorem}
Suppose $f$ satisfies \eqref{eq:fEqualsT0} with \eqref{eq:T0} and \eqref{eq:pointsInN}.
Suppose $c_0 = 1$, $d_0 = 0$, and $d_5 \ne 0$.
Suppose \eqref{eq:cond1}--\eqref{eq:cond4} are satisfied, $\Delta > 0$, and
\begin{equation}
-1 < c_2 < 1 - \frac{\sqrt{\Delta}}{2}.
\label{eq:sufficientInequality}
\end{equation}
Then there exists $k_{\rm min} \in \mathbb{Z}$ such that $f$ has an
asymptotically stable period-$(k+m)$ solution for all $k \in K(k_{\rm min})$.
\label{le:sufficientConditions}
\end{theorem}

Theorems \ref{le:necessaryConditions} and \ref{le:sufficientConditions} imply that the phenomenon of infinitely many asymptotically stable, single-round periodic solutions is codimension-three in the case $\lambda \sigma = -1$.  Specifically the three independent codimension-one conditions \eqref{eq:cond1}--\eqref{eq:cond3} need to hold, and the phenomenon indeed occurs if $\Delta > 0$ and \eqref{eq:sufficientInequality} holds.  In the case $\lambda \sigma = 1$ the phenomenon is codimension-four as we also require \eqref{eq:cond4}.  This condition is absent when $\lambda \sigma = -1$ because in this case the cubic terms in \eqref{eq:T0} are removable.

\section{Smooth parameter dependence}
\label{sec:parameters}
\setcounter{equation}{0}

Now suppose $f_\mu$ is a $C^\infty$ map on $\mathbb{R}^2$
with a $C^\infty$ dependence on a parameter $\mu \in \mathbb{R}^n$.
Let $\bO \in \mathbb{R}^n$ denote the origin in parameter space.
Suppose that for all $\mu$ in some region containing $\bO$,
the origin in phase space $(x,y) = (0,0)$ is a fixed point of $f_\mu$.
Let $\lambda = \lambda(\mu)$ and $\sigma = \sigma(\mu)$ be its associated eigenvalues
(these are $C^\infty$ functions of $\mu$) and suppose
\begin{align}
\lambda(\bO) &= \alpha, \label{eq:lambda} \\
\sigma(\bO) &= \frac{\chi_{\rm eig}}{\alpha}, \label{eq:sigma}
\end{align}
with $0 < \alpha < 1$ and $\chi_{\rm eig} \in \{ -1, 1 \}$.
With $\mu = \bO$ we have $|\lambda \sigma| = 1$, so as described above $T_0$ can be assumed to have the form \eqref{eq:T0}.
We now show we can assume $T_0$ has this form when the value of $\mu$ is sufficiently small.

\begin{lemma}
There exists a neighbourhood $\cN_{\rm param} \subset \mathbb{R}^n$ of $\bO$
and a $C^\infty$ coordinate change that puts $f_\mu$ in the form \eqref{eq:T0} for all $\mu \in \cN_{\rm param}$.
\label{le:T0}
\end{lemma}

\begin{proof}
Via a linear transformation $f_\mu$ can be transformed to
\begin{equation}
\begin{bmatrix} x \\ y \end{bmatrix} \mapsto
\begin{bmatrix}
\lambda x + \sum\limits_{i \ge 0, \,j \ge 0, \,i+j \ge 2} a_{ij} x^i y^j \\
\sigma y + \sum\limits_{i \ge 0, \,j \ge 0, \,i+j \ge 2} b_{ij} x^i y^j
\nonumber
\end{bmatrix},
\end{equation}
for some $a_{ij}, b_{ij} \in \mathbb{R}$.
It is a standard asymptotic matching exercise
to show that via an additional $C^\infty$ coordinate change we can achieve $a_{ij} = 0$ if $\lambda^{i-1} \sigma^j \ne 1$,
and $b_{ij} = 0$ if $\lambda^i \sigma^{j-1} \ne 1$. The remainder of the proof is based on this fact.

Assume the value of $\mu$ is small enough that \eqref{eq:eigAssumption} is satisfied.
Then $\lambda^p \sigma^q = 1$ is only possible with $p, q \ge 1$,
so a $C^\infty$ coordinate change can be performed to reduce the map to
\begin{equation}
\begin{bmatrix} x \\ y \end{bmatrix} \mapsto
\begin{bmatrix}
\lambda x + \sum\limits_{i \ge 2, \,j \ge 1} a_{ij} x^i y^j \\
\sigma y + \sum\limits_{i \ge 1, \,j \ge 2} b_{ij} x^i y^j
\nonumber
\end{bmatrix}.
\end{equation}
Since $|\lambda \sigma| = 1$ when $\mu = \bO$
we can assume $\mu$ is small enough that $\lambda^{p-1} \sigma \ne 1$ for all $p \ge 3$
and $\lambda \sigma^{q-1} \ne 1$ for all $q \ge 3$.
Consequently the map can further be reduced to
\begin{equation}
\begin{bmatrix} x \\ y \end{bmatrix} \mapsto
\begin{bmatrix}
\lambda x + a_{21} x^2 y + \sum\limits_{i \ge 3, \,j \ge 2} a_{ij} x^i y^j \\
\sigma y + b_{12} x y^2 + \sum\limits_{i \ge 2, \,j \ge 3} b_{ij} x^i y^j
\nonumber
\end{bmatrix},
\end{equation}
which can be rewritten as \eqref{eq:T0}.
\end{proof}

The product of the eigenvalues is
\begin{equation}
\lambda(\mu) \sigma(\mu) = \chi_{\rm eig} + \bn_{\rm eig}^{\sf T} \mu + \cO \left( \| \mu \|^2 \right),
\label{eq:nEig}
\end{equation}
where $\bn_{\rm eig}$ is the gradient of $\lambda \sigma$ evaluated at $\mu = \bO$.
The following result is an elementary application of the implicit function theorem.

\begin{lemma}
Suppose $\bn_{\rm eig} \ne \bO$.
Then $|\lambda(\mu) \sigma(\mu)| = 1$ on a $C^\infty$ codimension-one surface
intersecting $\mu = \bO$ and with normal vector $\bn_{\rm eig}$ at $\mu = \bO$ (as illustrated in Fig.~\ref{fig:Surf_Tang_Normal}).
\label{le:areaPreservingSurface}
\end{lemma}

\begin{figure}
\begin{center}
\includegraphics[width=8cm]{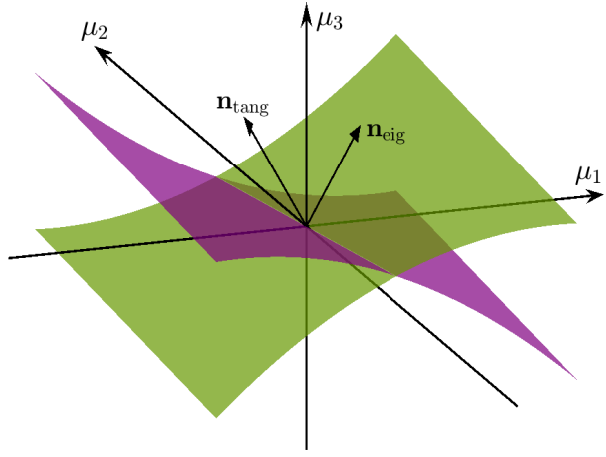}
\caption{A sketch of codimension-one surfaces of homoclinic tangencies (green) and where $\lambda(\mu) \sigma(\mu) = 1$ (purple).  The vectors ${\bf n}_{\rm tang}$ and ${\bf n}_{\rm eig}$, respectively, are normal to these surfaces at the origin $\mu = \bO$. 
\label{fig:Surf_Tang_Normal}
} 
\end{center}
\end{figure}


\section{The codimension-one surface of homoclinic tangencies}
\label{sec:tangency}
\setcounter{equation}{0}
In this section we describe the codimension-one surface of homoclinic tangencies that intersects $\mu = \bO$ where we will be assuming that a globally resonant homoclinic tangency occurs. 

Suppose \eqref{eq:pointsInN} is satisfied when $\mu = \bO$.
Write $f_\mu^m$ as \eqref{eq:T1} and suppose
\begin{align}
c_0(\bO) &= 1, \label{eq:c0} \\
d_0(\bO) &= 0, \label{eq:d0}
\end{align}
so that $f_\bO$ has an orbit homoclinic to the origin through $(x,y) = (1,0)$ and $(0,1)$.
Also suppose
\begin{align}
d_2(\bO) &= 0, \label{eq:d2} \\
d_5(\bO) &= d_{5,0} \ne 0, \label{eq:d5}
\end{align}
for a quadratic tangency.
Also write
\begin{equation}
d_0(\mu) = \bn_{\rm tang}^{\sf T} \mu + \cO \left( \| \mu \|^2 \right).
\label{eq:nTang}
\end{equation}

\begin{lemma}
Suppose \eqref{eq:pointsInN} is satisfied when $\mu = \bO$, \eqref{eq:c0}--\eqref{eq:d5} are satisfied, and  $\bn_{\rm tang} \ne \bO$.
Then $f_\mu$ has a quadratic homoclinic tangency to $(x,y) = (0,0)$ on a $C^\infty$ codimension-one surface intersecting $\mu = \bO$ and with normal vector $\bn_{\rm tang}$ at $\mu = \bO$.
Moreover, this tangency occurs at $(x,y) = (0,Y(\mu))$ where $Y$ is a $C^\infty$ function with $Y(\bO) = 1$.
\label{le:tangencySurface}
\end{lemma}

\begin{proof}
Let $T_{1,2}$ denote the second component of $T_1$ \eqref{eq:T1}.
The function
\begin{equation}
h(u;\mu) = \frac{\partial T_{1,2}}{\partial y}(0,1+u) = d_2(\mu) + 2 d_5(\mu) u + \cO(2),
\nonumber
\end{equation}
is a $C^\infty$ function of $u \in \mathbb{R}$ and $\mu$.
Since $h(0; \mathbf{0}) = 0$ by \eqref{eq:d2} and $\frac{\partial h}{\partial u}(0;\mathbf{0}) \ne 0$ by \eqref{eq:d5},
the implicit function theorem implies there exists a $C^\infty$ function $u_{\rm tang}(\mu)$ such that
$h(u_{\rm tang}(\mu);\mu) = 0$ locally.

By construction, the unstable manifold of $(x,y) = (0,0)$ is tangent to the $x$-axis at $T_1(0,1+u_{\rm tang})$.
Moreover this tangency is quadratic by \eqref{eq:d5}.
Thus a homoclinic tangency occurs if
$T_{1,2}(0,1+u_{\rm tang}) = d_0 + d_2 u_{\rm tang} + d_5 u_{\rm tang}^2 + \cO(3) = 0$.
This function is $C^\infty$ and
\begin{equation}
T_{1,2}(0,1+u_{\rm tang}) = \bn_{\rm tang}^{\sf T} \mu + \cO \left( \| \mu \|^2 \right).
\nonumber
\end{equation}
Since $\bn_{\rm tang} \ne \bO$ the result follows from another application of the implicit function theorem.
\end{proof}

\section{Sequences of saddle-node and period-doubling bifurcations}
\label{sec:mainResult}
\setcounter{equation}{0}

Suppose
\begin{equation}
d_1(\bO) = \chi,
\label{eq:d1}
\end{equation}
where $\chi \in \{ -1, 1 \}$.
Write
\begin{equation}
\begin{aligned}
a_1(\bO) &= a_{1,0} \,,
&\qquad c_1(\bO) &= c_{1,0} \,,
&\qquad d_3(\bO) &= d_{3,0} \,, \\
b_1(\bO) &= b_{1,0} \,,
&\qquad c_2(\bO) &= c_{2,0} \,,
&\qquad d_4(\bO) &= d_{4,0} \,,
\end{aligned}
\label{eq:other}
\end{equation}
and, recalling \eqref{eq:Delta}, let
\begin{equation}
\Delta_0 = \left( 1 - c_{2,0} - \chi d_{4,0} \right)^2 - 4 d_{5,0} (d_{3,0} + \chi c_{1,0}).
\label{eq:Delta0}
\end{equation}

\begin{theorem}
Suppose $f_\mu$ satisfies \eqref{eq:lambda}, \eqref{eq:sigma}, \eqref{eq:c0}--\eqref{eq:d5}, $a_{1,0} + b_{1,0} = 0$, $\Delta_0 > 0$, and $-1 < c_{2,0} < 1 - \frac{\sqrt{\Delta_0}}{2}$.
Let $\bv \in \mathbb{R}^n$.
If $\bn_{\rm tang}^{\sf T} \bv \ne 0$
then there exists $k_{\rm min} \in \mathbb{Z}$ such that for all $k \in K(k_{\rm min})$
there exist $\ee_k^- <0$ and $\ee_k^+ > 0$ with $\ee_k^\pm = \cO \left( \alpha^{2 k} \right)$
such that $f$ has an asymptotically stable period-$(k+m)$ solution for all $\mu = \ee \bv$ with $\ee_k^- < \ee < \ee_k^+$.
Moreover, one sequence, $\{ \ee_k^- \}$ or $\{ \ee_k^+ \}$, corresponds to saddle-node bifurcations of the periodic solutions, the other to period-doubling bifurcations.
If instead $\bn_{\rm tang}^{\sf T} \bv = 0$ and $\bn_{\rm eig}^{\sf T} \bv \ne 0$ the same results hold now $\ee_k^\pm = \cO \left( \frac{\alpha^k}{k} \right)$.
\label{th:main}
\end{theorem}

\section{Proof of the main result}
\label{sec:mainProof}
\setcounter{equation}{0}

To prove Theorem \ref{th:main} we use the following lemma which is proved in Appendix \ref{app:proof1} by carefully estimating the error terms in (7.1)..
If $\chi_{\rm eig} = -1$ then \eqref{eq:T0k} is true if $a_1 = b_1 = 0$
(and this can be proved in the same fashion).

\begin{lemma}
Suppose $\chi_{\rm eig} = 1$ and $\mu = \cO \left( \alpha^k \right)$.
If $|x-1|$ and $\left| \alpha^{-k} y - 1 \right|$ are sufficiently small
for all sufficiently large values of $k$, then
\begin{equation}
T_0^k(x,y) = \begin{bmatrix}
\lambda^k x \left( 1 + k a_1 x y + \cO \left( k^2 \alpha^{2 k} \right) \right) \\
\sigma^k y \left( 1 + k b_1 x y + \cO \left( k^2 \alpha^{2 k} \right) \right)
\end{bmatrix}.
\label{eq:T0k}
\end{equation}
\label{le:T0k}
\end{lemma}

\begin{proof}[Proof of Theorem \ref{th:main}]
\myStep{1}{Coordinate changes to distinguish the surface of homoclinic tangencies.}
First we perform two coordinate changes on parameter space.
There exists an $n \times n$ orthogonal matrix $A$ such that after $\mu$ is replaced with $A \mu$,
and $\mu$ is scaled,
we have $\bn_{\rm tang}^{\sf T} = [1,0,\ldots,0]$ --- the first coordinate vector of $\mathbb{R}^n$.
Then $d_0(\mu) = \mu_1 + \cO \left( \| \mu \|^2 \right)$.
Second, for convenience, we apply a near-identity transformation to remove the higher order terms, resulting in
\begin{equation}
d_0(\mu) = \mu_1 \,.
\label{eq:d0Exp}
\end{equation}
These coordinate changes do not alter the signs of the dot products ${\bf n}_{\rm tang}^{\sf T} {\bf v}$ and ${\bf n}_{\rm eig}^{\sf T} {\bf v}$.  Now write
\begin{align}
c_0(\mu) &= 1 + \sum_{i=1}^n p_i \mu_i + \cO \left( \| \mu \|^2 \right), \label{eq:c0Exp} \\
d_1(\mu) &= \chi + \sum_{i=1}^n q_i \mu_i + \cO \left( \| \mu \|^2 \right), \label{eq:d1Exp} \\
d_2(\mu) &= \sum_{i=1}^n r_i \mu_i + \cO \left( \| \mu \|^2 \right), \label{eq:d2Exp}
\end{align}
\begin{align}
\lambda(\mu) &= \alpha + \sum_{i=1}^n s_i \mu_i + \cO \left( \| \mu \|^2 \right), \label{eq:lambdaExp} \\
\sigma(\mu) &= \frac{\chi_{\rm eig}}{\alpha} + \sum_{i=1}^n t_i \mu_i + \cO \left( \| \mu \|^2 \right), \label{eq:sigmaExp}
\end{align}
where $p_i,\ldots,t_i \in \mathbb{R}$ are constants.
\myStep{2}{Apply a $k$-dependent scaling to $\mu$.}  In view of the coordinate changes applied in the previous step, the surface of homoclinic surfaces of Lemma 5.1 is tangent to the $\mu_1 = 0$ coordinate hyperplane.  In order to show that bifurcation values scale like $|\alpha|^{2 k}$ if we adjust the value of $\mu$ in a direction transverse to this surface, and, generically, scale like $\frac{|\alpha|^k}{k}$ otherwise, we scale the components of $\mu$ as follows:
\begin{equation}
\mu_i = \begin{cases}
\alpha^{2 k} \tilde{\mu}_i \,, & i = 1, \\
\alpha^k \tilde{\mu}_i \,, & i \ne 1.
\end{cases}
\label{eq:muScaling}
\end{equation}
Below we will see that the resulting asymptotic expansions are consistent and this will justify \eqref{eq:muScaling}.  Notice that $\tilde{\mu}_1$-terms are higher order than $\tilde{\mu}_i$-terms, for $i \ne 1$.
For example \eqref{eq:lambdaExp} now becomes $\lambda = \alpha + \sum_{i=2}^n s_i \tilde{\mu}_i \alpha^k + \cO \left( |\alpha|^{2 k} \right)$.
Further, let $k$ be such that $\lambda(\bO)^k \sigma(\bO)^k = d_1(\bO)$, that is $\chi_{\rm eig}^k = \chi$.
Then from \eqref{eq:lambdaExp}--\eqref{eq:muScaling} we obtain
\begin{align}
\lambda^k &= \alpha^k \left( 1 + \frac{k}{\alpha} \sum_{i=2}^n s_i \tilde{\mu}_i \alpha^k
+ \cO \left( k^2 |\alpha|^{2 k} \right) \right), \label{eq:lambdak} \\
\sigma^k &= \frac{\chi}{\alpha^k} \left( 1 + k \alpha \chi_{\rm eig} \sum_{i=2}^n t_i \tilde{\mu}_i \alpha^k
+ \cO \left( k^2 |\alpha|^{2 k} \right) \right). \label{eq:sigmak} \\
\end{align}

\myStep{3}{Calculate one point of each periodic solution.}  One point of a single-round periodic solution is a fixed point of $T_0^k \circ T_1$. We look for fixed points $\left( x^{(k)}, y^{(k)} \right)$ of $T_0^k \circ T_1$ of the form
\begin{equation}
\begin{split}
x^{(k)} &= \alpha^k \left( 1 + \phi_k \alpha^k + \cO \left( k^2 |\alpha|^{2 k} \right) \right), \\ 
y^{(k)} &= 1 + \psi_k \alpha^k + \cO \left( k^2 |\alpha|^{2 k} \right). 
\label{eq:xkyk}
\end{split}
\end{equation}
By substituting \eqref{eq:xkyk} into \eqref{eq:T1}
and the above various asymptotic expressions for the coefficients in $T_1$, we obtain
\begin{equation}
T_1 \left( x^{(k)}, y^{(k)} \right) =
\begin{bmatrix}

1 + \left( c_{1,0} + c_{2,0} \psi_k + \sum_{i=2}^n p_i \tilde{\mu}_i \right) \alpha^k + \cO \left( k^2 |\alpha|^{2 k} \right)\\

\begin{aligned}
\chi \alpha^k + \left( \tilde{\mu}_1 + d_{3,0} + d_{4,0} \psi_k + d_{5,0} \psi_k^2 + \chi \phi_k
+ \sum_{i=2}^n \left( q_i + r_i \psi_k \right) \tilde{\mu}_i \right) \alpha^{2 k}&\\[2mm]
+ \cO \left( k^2 |\alpha|^{3 k} \right)&
\end{aligned}
\end{bmatrix}.
\nonumber
\end{equation}
Then by \eqref{eq:T0k},
{\footnotesize
\begin{equation}
\left( T_0^k \circ T_1 \right) \left( x^{(k)}, y^{(k)} \right) =
\begin{bmatrix}
\begin{aligned}
\alpha^k + \bigg( a_{1,0} \chi + \frac{1}{\alpha} \sum_{i=2}^n s_i \tilde{\mu}_i \bigg) k \alpha^{2 k}
+ \bigg( c_{1,0} + c_{2,0} \psi_k + \sum_{i=2}^n p_i \tilde{\mu}_i \bigg) \alpha^{2 k} &\\[2mm]
+ \cO \left( k^2 |\alpha|^{3 k} \right) &
\end{aligned}\\
\begin{aligned}
1 + \bigg( b_{1,0} \chi + \alpha \chi_{\rm eig} \sum_{i=2}^n t_i \tilde{\mu}_i \bigg) k \alpha^k
+ \bigg( \chi \tilde{\mu}_1 + \chi d_{3,0} + \chi d_{4,0} \psi_k &\\
+ \chi d_{5,0} \psi_k^2 + \phi_k + \chi \sum_{i=2}^n \left( q_i + r_i \psi_k \right) \tilde{\mu}_i \bigg) \alpha^k + \cO \left( k^2 |\alpha|^{2 k} \right) &
\end{aligned}
\end{bmatrix}.
\label{eq:T0kT1xkykun}
\end{equation}
}
By matching \eqref{eq:xkyk} and \eqref{eq:T0kT1xkykun} and eliminating $\phi_k$
we obtain the following expression that determines the possible values of $\psi_k$:
\begin{equation}
\chi d_{5,0} \psi_k^2 - P \psi_k + Q=0,
\label{eq:psikQuadratic}
\end{equation}
where
\begin{align}
P &= 1 - c_{2,0} - \chi d_{4,0} - \chi \sum_{i=2}^n r_i \tilde{\mu}_i \,, \label{eq:P} \\
Q &= \chi \tilde{\mu}_1 + c_{1,0} + \chi d_{3,0} + \sum_{i=2}^n \left( p_i + \chi q_i \right) \tilde{\mu}_i
+ \sum_{i=2}^n \left( \frac{s_i}{\alpha} + \alpha \chi_{\rm eig} t_i \right) \tilde{\mu}_i k, \label{eq:Q}
\end{align}
and we have also used $a_{1,0} + b_{1,0} = 0$.
Of the two solutions to \eqref{eq:psikQuadratic}, the one that yields an asymptotically stable solution when $\mu = \bO$ is
\begin{equation}
\psi_k = \frac{1}{2 \chi d_{5,0}} \left( P - \sqrt{P^2 - 4 \chi d_{5,0} Q} \right).
\label{eq:psik}
\end{equation}
Note that this solution exists and is real-valued for sufficiently small values of $\mu$ because when $\mu = \bO$
the discriminant is $P^2 - 4 \chi d_{5,0} Q = \Delta_0 > 0$.

\myStep{4}{Stability of the periodic solution.} By using \eqref{eq:T1}, \eqref{eq:T0k}, \eqref{eq:lambdak}, \eqref{eq:sigmak}, and \eqref{eq:xkyk},
\begin{equation}
\rD \left( T_0^k \circ T_1 \right) \left( x^{(k)}, y^{(k)} \right) =
\begin{bmatrix}
\cO \left( |\alpha|^k \right) &
c_{2,0} \alpha^k + \cO \left( k |\alpha|^{2 k} \right) \\
\frac{1}{\alpha^k} \left( 1 + \cO \left( k |\alpha|^k \right) \right) &
\chi d_{4,0} + \chi \sum_{i=2}^n r_i \tilde{\mu}_i + 2 \chi d_{5,0} \psi_k + \cO \left( k |\alpha|^k \right)
\end{bmatrix}.
\label{eq:Jacobian}
\end{equation}
Let $\tau_k$ and $\delta_k$ denote the trace and determinant of this matrix, respectively.
By \eqref{eq:P}, \eqref{eq:psik}, and \eqref{eq:Jacobian} we obtain
\begin{align}
\tau_k &= 1 - c_{2,0} - \sqrt{P^2 - 4 \chi d_{5,0} Q} + \cO \left( k |\alpha|^k \right), \label{eq:tauk} \\
\delta_k &= -c_{2,0} + \cO \left( k |\alpha|^k \right). \label{eq:deltak}
\end{align}
By substituting $\mu = \bO$ into \eqref{eq:tauk}
we obtain $\tau_k = 1 - c_{2,0} - \sqrt{\Delta_0}$.
It immediately follows from the assumption $-1 < c_{2,0} < 1 - \frac{\sqrt{\Delta_0}}{2}$
that the periodic solution is asymptotically stable for sufficiently large values of $k$.

\myStep{5}{The generic case ${\bf n}_{\rm tang}^{\sf T} {\bf v} \ne 0$.} Now suppose $\bn_{\rm tang}^{\sf T} \bv \ne 0$, that is, $v_1 \ne 0$ (in view of the earlier coordinate change).
Write $\mu = \ee \bv$ and $\ee = \tilde{\ee} \alpha^{2 k}$.
By \eqref{eq:muScaling}, $\tilde{\mu}_1 = \tilde{\ee} v_1$ and $\tilde{\mu}_i = \tilde{\ee} v_i \alpha^k$ for $i \ne 1$.
Then by \eqref{eq:P} and \eqref{eq:Q}, $P = 1 - c_{2,0} - \chi d_{4,0} + \cO \left( |\alpha|^k \right)$
and $Q = c_{1,0} + \chi d_{3,0} + \chi \tilde{\ee} v_1 + \cO \left( |\alpha|^k \right)$, so
\begin{equation}
P^2 - 4 \chi d_{5,0} Q = \left( 1 - c_{2,0} - \chi d_{4,0} \right)^2
- 4 \chi d_{5,0} \left( c_{1,0} + \chi d_{3,0} + \chi \tilde{\ee} v_1 \right) + \cO \left( |\alpha|^k \right).
\label{eq:DeltaCase1}
\end{equation}
Since $v_1 \ne 0$ and $d_{5,0} \ne 0$ we can solve $\delta_k - \tau_k + 1 = 0$ for $\tilde{\ee}$
(formally this is achieved via the implicit function theorem) and the solution is
\begin{equation}
\tilde{\ee}_{\rm SN} = \frac{\Delta_0}{4 d_{5,0} v_1} + \cO \left( k |\alpha|^k \right).
\label{eq:tildeeeSNCase1}
\end{equation}
Also we can use \eqref{eq:DeltaCase1} to solve $\delta_k + \tau_k + 1 = 0$ for $\tilde{\ee}$:
\begin{equation}
\tilde{\ee}_{\rm PD} = \frac{\Delta_0 - 4 \left( 1 - c_{2,0} \right)^2}{4 d_{5,0} v_1} + \cO \left( k |\alpha|^k \right).
\label{eq:tildeeePDCase1}
\end{equation}
Since $\tilde{\ee}_{\rm SN}$ and $\tilde{\ee}_{\rm PD}$ evidently have opposite signs, this completes the proof in the first case.

\myStep{6}{The degenerate case ${\bf n}_{\rm tang}^{\sf T} {\bf v} = 0$.}  Now suppose $v_1 = 0$ and $\bn_{\rm eig}^{\sf T} \bv \ne 0$.
Again write $\mu = \ee \bv$ but now write $\ee = \frac{\tilde{\ee} \alpha^k}{k}$.
By \eqref{eq:muScaling}, $\tilde{\mu}_1 = 0$ and $\tilde{\mu}_i = \frac{\tilde{\ee} v_i}{k}$ for $i \ne 1$.

We first evaluate $\bn_{\rm eig}^{\sf T} \bv$.
By multiplying \eqref{eq:lambdaExp} and \eqref{eq:sigmaExp} and comparing the result to \eqref{eq:nEig}
we see that the $i^{\rm th}$ element of $\bn_{\rm eig}$
is $\frac{\chi_{\rm eig} s_i}{\alpha} + \alpha t_i$.
Then since $v_1 = 0$ we have
$\bn_{\rm eig}^{\sf T} \bv = \sum_{i=2}^n \left( \frac{\chi_{\rm eig} s_i}{\alpha} + \alpha t_i \right) v_i$.
Further, $\mu = \ee \bv$ and $\ee = \frac{\tilde{\ee} \alpha^k}{k}$,
\begin{equation}
\chi_{\rm eig} \tilde{\ee} \bn_{\rm eig}^{\sf T} \bv
= \sum_{i=2}^n \left( \frac{s_i}{\alpha} + \alpha \chi_{\rm eig} t_i \right) \tilde{\mu}_i k,
\nonumber
\end{equation}
which is a term appearing in \eqref{eq:Q}.
So by \eqref{eq:P} and \eqref{eq:Q}, $P = 1 - c_{2,0} - \chi d_{4,0} + \cO \left( \frac{1}{k} \right)$ and $Q = c_{1,0} + \chi d_{3,0} + \frac{\tilde{\ee} \chi_{\rm eig} \bn_{\rm eig}^{\sf T} \bv}{k} + \cO \left( \frac{1}{k} \right)$.
By solving $\delta_k - \tau_k + 1 = 0$,
\begin{equation}
\tilde{\ee}_{\rm SN} = \frac{\Delta_0}{4 d_{5,0} \chi \chi_{\rm eig} \bn_{\rm eig}^{\sf T} \bv} + \cO \left( \frac{1}{k} \right),
\label{eq:tildeeeSNCase2}
\end{equation}
and by solving $\delta_k - \tau_k + 1 = 0$,
\begin{equation}
\tilde{\ee}_{\rm PD} = \frac{\Delta_0 - 4 \left( 1 - c_{2,0} \right)^2}{4 d_{5,0} \chi \chi_{\rm eig} \bn_{\rm eig}^{\sf T} \bv} + \cO \left( \frac{1}{k} \right).
\label{eq:tildeeePDCase2}
\end{equation}
\end{proof}
As in the previous case, $\tilde{\ee}_{\rm SN}$ and $\tilde{\ee}_{\rm PD}$ have opposite signs, and this completes the proof in the second case.  Notice how the assumptions we have made ensure the denominators of \eqref{eq:tildeeeSNCase2} and \eqref{eq:tildeeePDCase2} are nonzero.

\section{A comparison to numerically computed bifurcation values.}
\label{sec:numerics}
\setcounter{equation}{0}

Here we extend the example given in \cite{MuMcSi21} and
illustrate the results with the following $C^1$ family of maps
\begin{equation}
f(x,y) = \begin{cases}
U_0(x,y), & y \le \frac{2 \alpha + 1}{3}, \\
(1 - r(y)) U_0(x,y) + r(y) U_1(x,y), & \frac{2 \alpha + 1}{3} \le y \le \frac{\alpha + 2}{3}, \\
U_1(x,y), & y \ge \frac{\alpha + 2}{3},
\end{cases}
\label{eq:toy}
\end{equation}
where
\begin{align}
U_0(x,y) &= \begin{bmatrix}
(\alpha + \mu_2) x \big( 1 + (a_{1,0} + \mu_4) x y \big) \\
\frac{1}{\alpha} y (1 - a_{1,0} x y)
\end{bmatrix},\\
\label{eq:U0}
U_1(x,y) &= \begin{bmatrix}
1 + c_{2,0} (y-1) \\
\mu_1 + (1 + \mu_3) x + d_{5,0} (y-1)^2
\end{bmatrix},\\
r(y) &= s \left( \frac{y - h_0}{h_1 - h_0} \right),
\end{align}
where
\begin{equation}
s(z) = 3 z^2 - 2 z^3,\\
\label{eq:s}
\end{equation}
Below we fix
\begin{equation}
\begin{split}
\alpha &= 0.8, \\
a_{1,0} &= 0.2, \\
c_{2,0} &= -0.5, \\
d_{5,0} &= 1,
\end{split}
\label{eq:toyParameters}
\end{equation}
and vary $\mu = (\mu_1,\mu_2,\mu_3,\mu_4) \in \mathbb{R}^4$.

\begin{figure}
\begin{center}
\includegraphics[width=12cm]{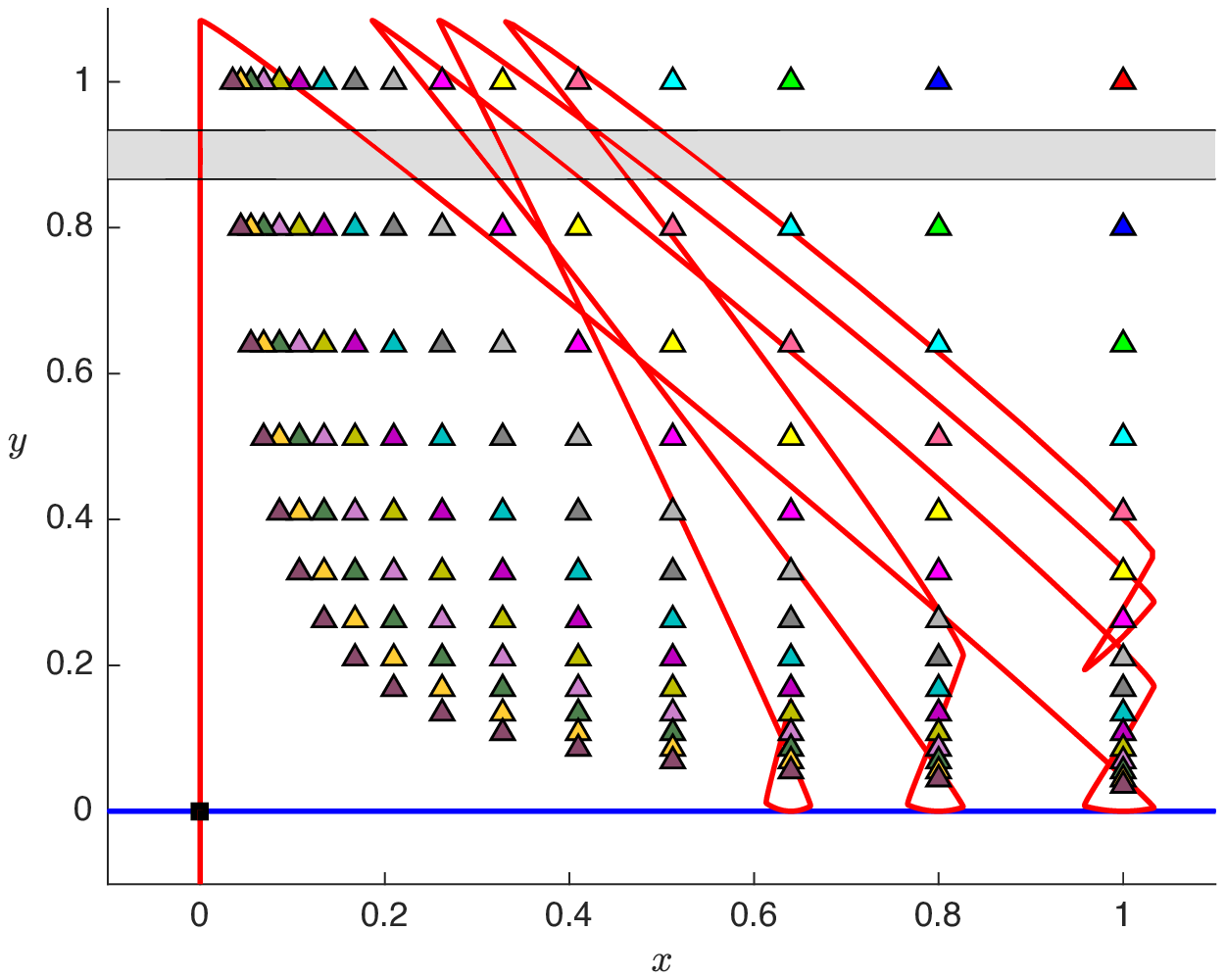}\\
\includegraphics{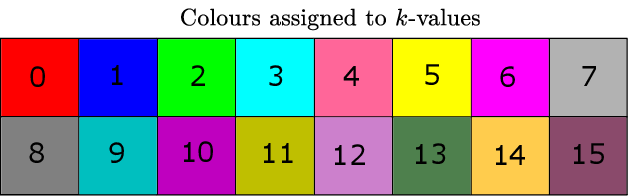}
\caption{A phase portrait of \eqref{eq:toy} with \eqref{eq:toyParameters} and $\mu = \bO$.  The shaded horizontal strip is where the middle component of \eqref{eq:toy} applies.  We show parts of the stable and unstable manifolds of $(x,y) = (0,0)$.  Note the unstable manifold has very high curvature at $(x,y) \approx (0,1.1)$ because \eqref{eq:toy} is highly nonlinear in the horizontal strip.  For the given parameter values \eqref{eq:toy} has an asymptotically stable, single-round periodic solutions of period $k+1$ for all $k \ge 1$.  These are shown for $k = 1,2,\ldots,15$; different colours correspond to different values of $k$.  The map also has an asymptotically stable fixed point at $(x,y) = (1,1)$.
\label{fig:HTSRPUNF}
} 
\end{center}
\end{figure}

With $\mu = \bO$, \eqref{eq:toy} satisfies the conditions of Theorem \ref{le:sufficientConditions}. In particular $\Delta_0 = 2.25$ so $\Delta_0 > 0$ and $-1 < c_{2,0} < 1 - \frac{\Delta_0}{2}$.  Therefore \eqref{eq:toy} has an asymptotically stable single-round periodic solution for all sufficiently large values of $k$.
In fact these exist for all $k \ge 1$, see Fig.~\ref{fig:HTSRPUNF}, plus there exists an asymptotically stable fixed point at $(x,y) = (1,1)$ that can be interpreted as corresponding to $k=0$.\\
\begin{figure}[b!]
\begin{center}
\setlength{\unitlength}{1cm}
\begin{picture}(14,5.1)
\put(0,0){\includegraphics[width=6.8cm]{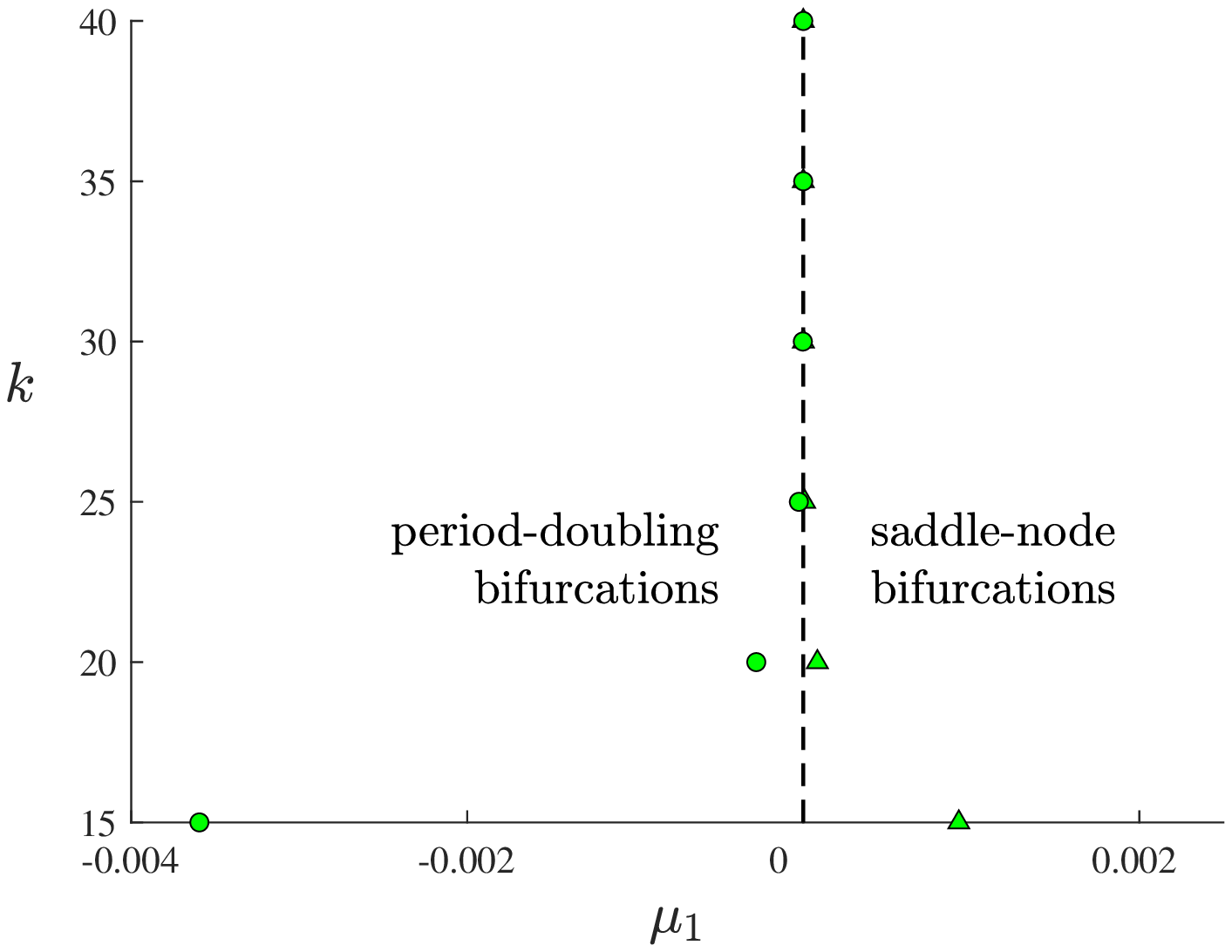}}
\put(7.2,0){\includegraphics[width=6.8cm]{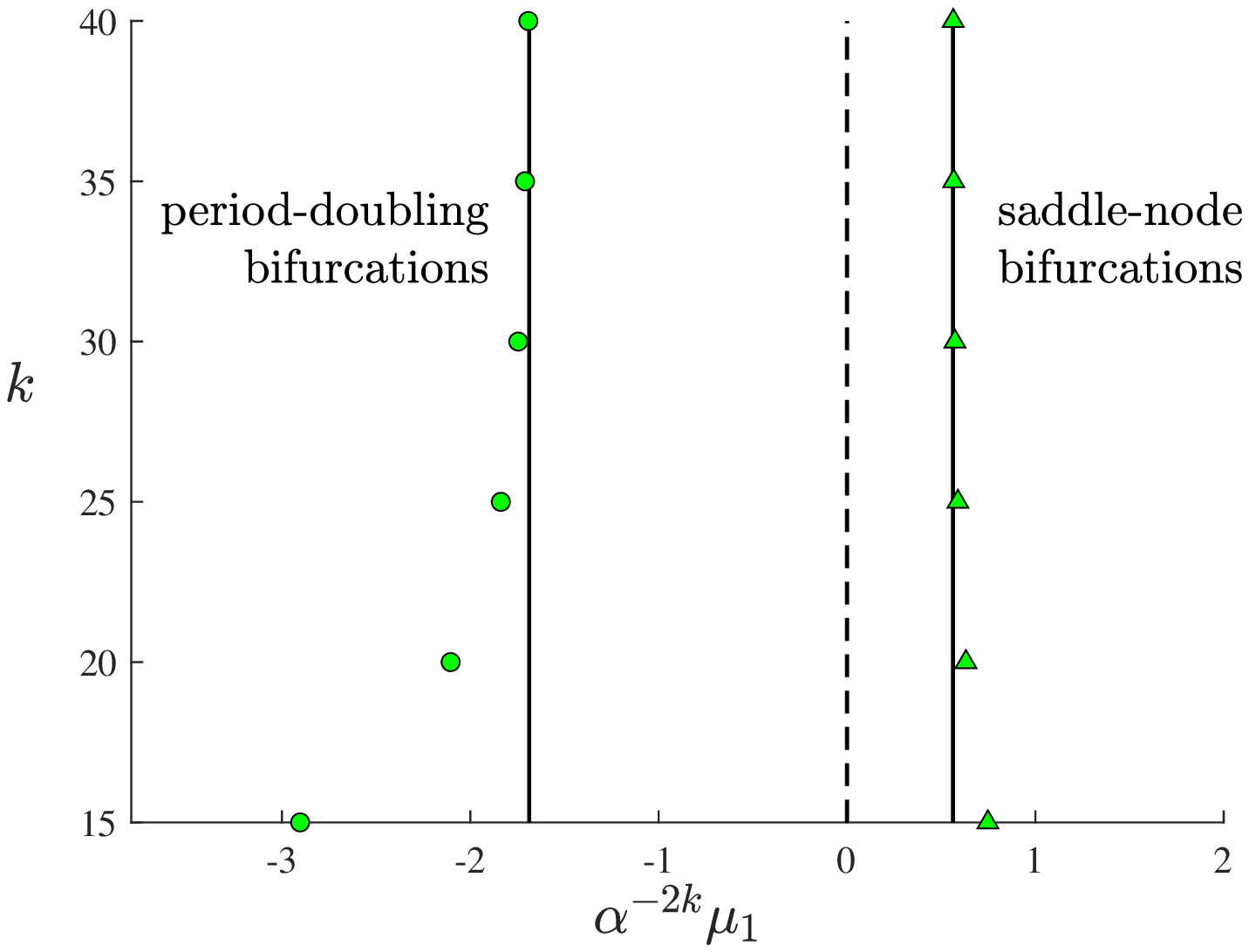}}
\put(0,5.1){\small \bf a)}
\put(7.2,5.1){\small \bf b)}
\end{picture}
\caption{Panel (a) is a numerically computed bifurcation diagram of \eqref{eq:toy} with \eqref{eq:toyParameters} and $\mu_2 = \mu_3 = \mu_4 = 0$.  The triangles [circles] are saddle-node [period-doubling] bifurcations of single-round periodic solutions of period $k+1$.  Panel (b) shows the same points but with the horizontal axis scaled in such a way that the asymptotic approximations to these bifurcations, given by the leading-order terms in \eqref{eq:tildeeeSNCase1} and \eqref{eq:tildeeePDCase1}, appear as vertical lines.
\label{fig:A}
} 
\end{center}
\end{figure}
In this remainder of this section we study how the infinite coexistence is destroyed by varying each the components of $\mu$ from zero in turn.  We identify saddle-node and period-doubling bifurcations numerically and compare these to our above asymptotic results. First, by varying the value of $\mu_1$ from zero we destroy the homoclinic tangency.
Indeed in \eqref{eq:nTang} we have ${\bf n}_{\rm tang}^{\sf T} = [1,0,0,0]$.
Thus if we fix $\mu_2 = \mu_3 = \mu_4 = 0$ and vary the value of $\mu_1$,
by Theorem \ref{th:main} there must exist sequences of saddle-node and period-doubling bifurcations
occurring at values of $\mu_1$ that are asymptotically proportional to $\alpha^{2 k}$.
Fig.~\ref{fig:A}a shows the bifurcation values (obtained numerically) for six different values of $k$.
We have designed \eqref{eq:toy} so that it satisfies \eqref{eq:d0Exp}.  Consequently the formulas \eqref{eq:tildeeeSNCase1} and \eqref{eq:tildeeePDCase1} for the bifurcation values can be applied directly.  In panel (b) we observe that the numerically computed bifurcation values indeed  converge to their leading-order approximations.

\begin{figure}[b!]
\begin{center}
\setlength{\unitlength}{1cm}
\begin{picture}(14,5.1)
\put(0,0){\includegraphics[width=6.8cm]{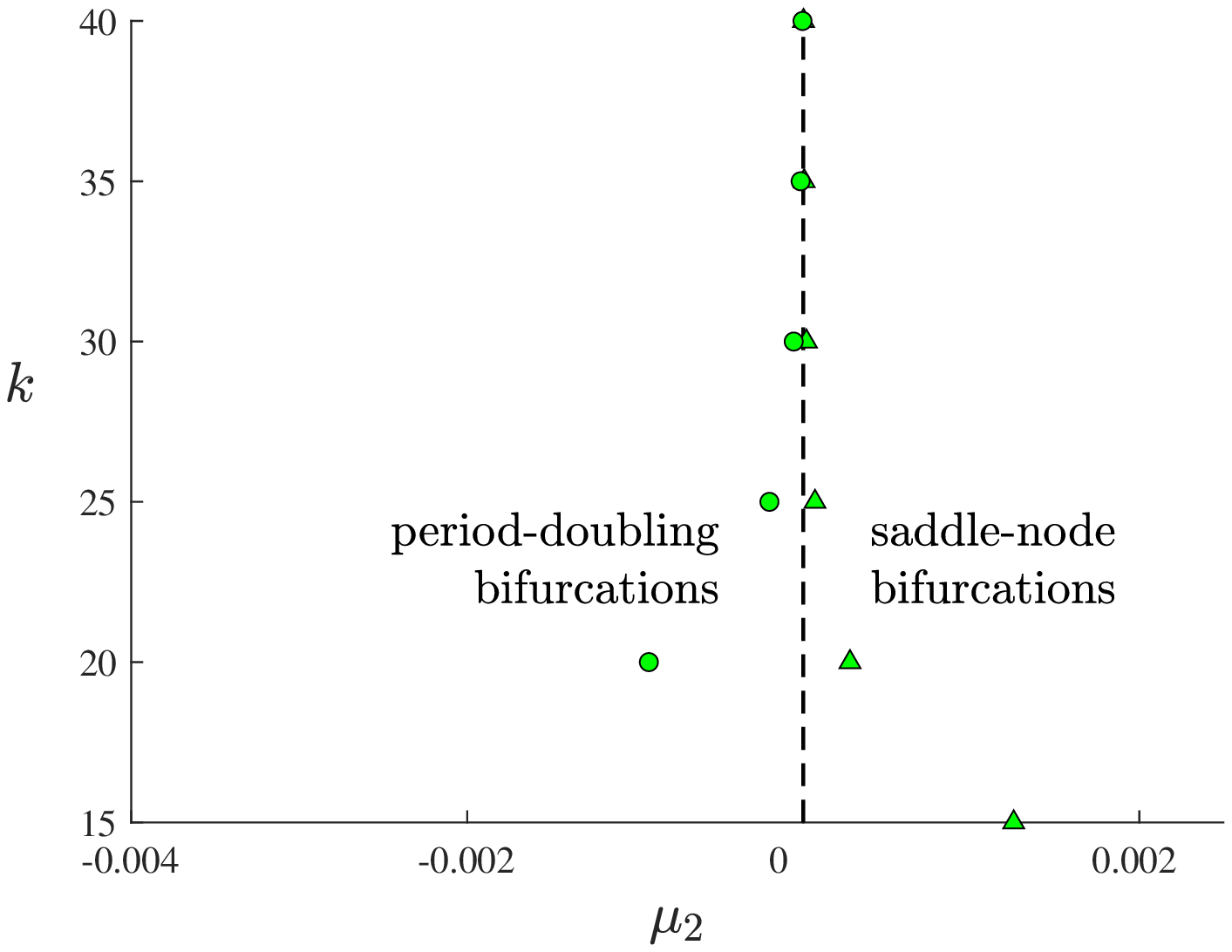}}
\put(7.2,0){\includegraphics[width=6.8cm]{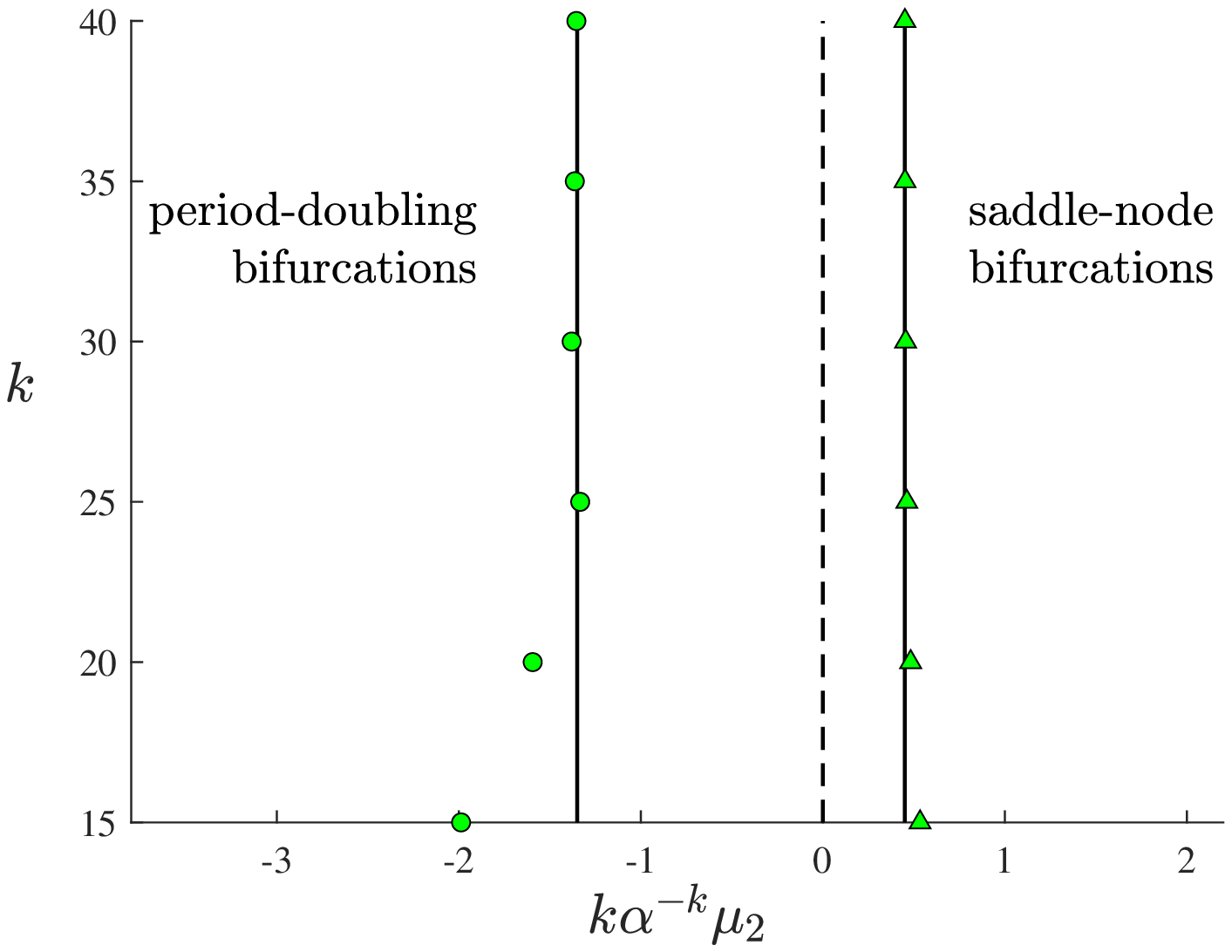}}
\put(0,5.1){\small \bf a)}
\put(7.2,5.1){\small \bf b)}
\end{picture}
\caption{Panel (a) is a numerically computed bifurcation diagram of \eqref{eq:toy} with \eqref{eq:toyParameters} and $\mu_1 = \mu_3 = \mu_4 = 0$.  Panel (b) shows convergence to the leading-order terms of \eqref{eq:tildeeeSNCase2} and \eqref{eq:tildeeePDCase2}.
\label{fig:B}
} 
\end{center}
\end{figure}

We now fix $\mu_1 = \mu_3 = \mu_4=0$ and vary the value of $\mu_2$.
This parameter variation alters the product of the eigenvalues associated with the origin.
Specifically ${\bf n}_{\rm eig}^{\sf T} = \left[ 0,\frac{1}{\alpha},0,0 \right]$ in \eqref{eq:nEig}
so by Theorem \ref{th:main} the bifurcation values are asymptotically proportional to $\frac{\alpha^k}{k}$.
In Fig.~\ref{fig:B} we see the numerically computed bifurcation values converging to their leading order approximations \eqref{eq:tildeeeSNCase2} and \eqref{eq:tildeeePDCase2}.

\begin{figure}[b!]
\begin{center}
\setlength{\unitlength}{1cm}
\begin{picture}(14,5.1)
\put(0,0){\includegraphics[width=6.8cm]{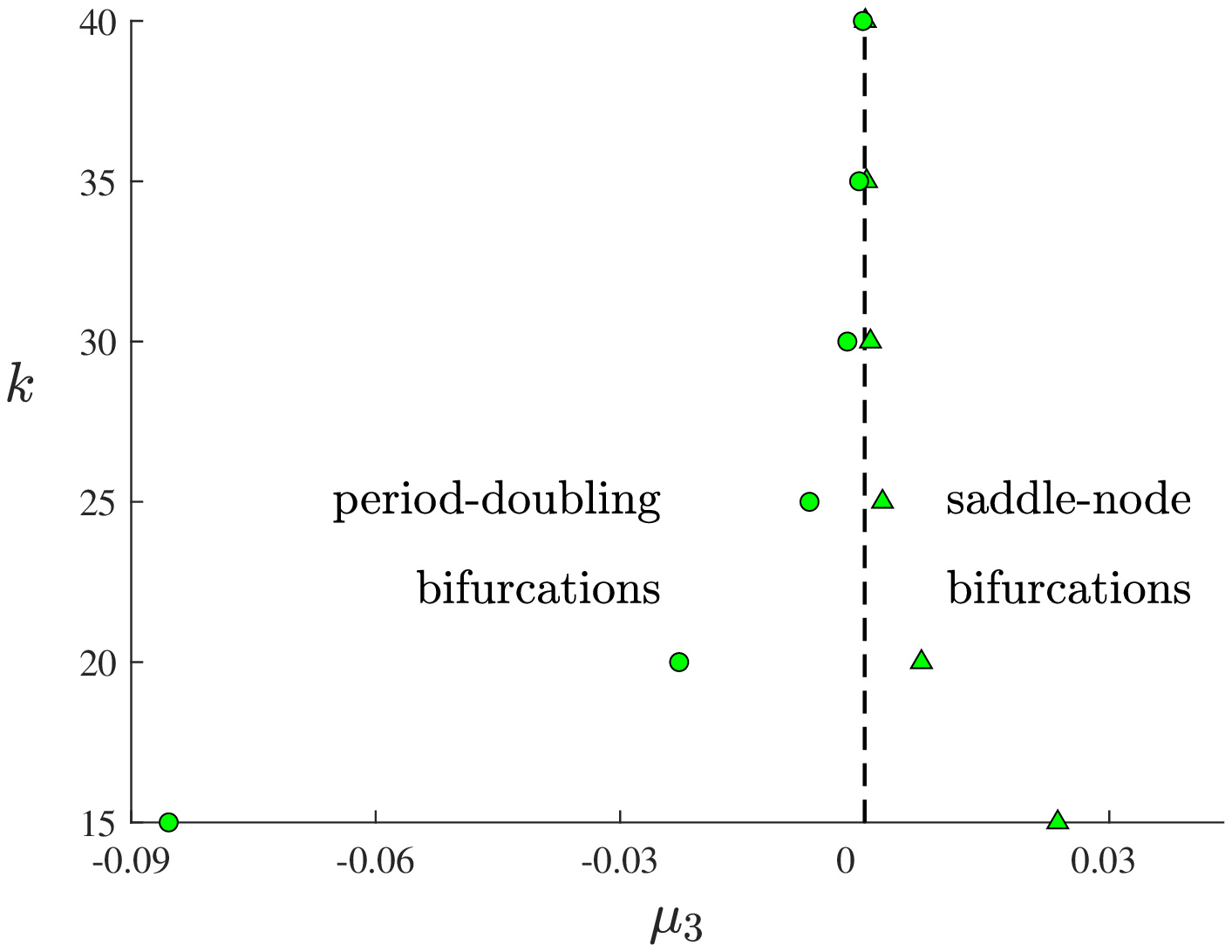}}
\put(7.2,0){\includegraphics[width=6.8cm]{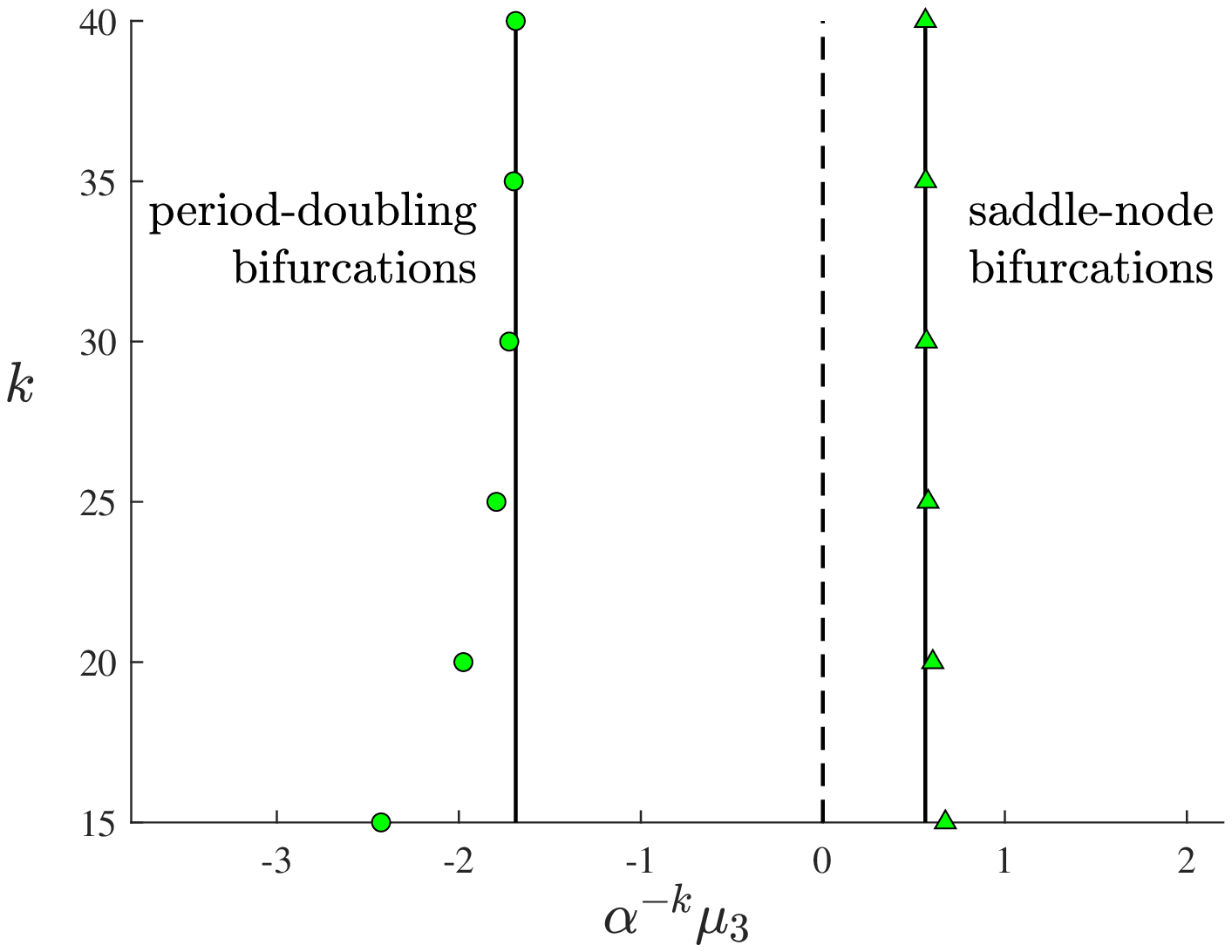}}
\put(0,5.1){\small \bf a)}
\put(7.2,5.1){\small \bf b)}
\end{picture}
\caption{Panel (a) is a numerically computed bifurcation diagram of \eqref{eq:toy} with \eqref{eq:toyParameters} and $\mu_1 = \mu_2 = \mu_4 = 0$.  Panel (b) shows convergence to the leading-order terms of \eqref{eq:CO3SN} and \eqref{eq:CO3PD}.
\label{fig:C}
} 
\end{center}
\end{figure}

Next we fix $\mu_1 = \mu_2 = \mu_4 = 0$ and vary the value of $\mu_3$ which breaks the global resonance condition.
Here ${\bf n}_{\rm tang}^{\sf T} {\bf v} = 0$ and ${\bf n}_{\rm eig}^{\sf T} {\bf v} = 0$ so Theorem \eqref{th:main} does not apply.  But by performing calculations analogous to those given above in the proof of Theorem \ref{th:main} directly to the map \eqref{eq:toy}, 
we obtain the following expressions for the saddle-node and period-doubling bifurcation values
\begin{equation}
    \epsilon_{\rm SN} = \frac{(1-c_{2,0})^2}{4d_{5,0}} \alpha^{k} + \mathcal{O}(|\alpha|^{2k}),
    \label{eq:CO3SN}
\end{equation}

\begin{equation}
    \epsilon_{\rm PD} = -\frac{3(1-c_{2,0})^2}{4d_{5,0}} \alpha^{k}+ \mathcal{O}(|\alpha|^{2k}).
    \label{eq:CO3PD}
\end{equation}
Fig.~\ref{fig:C} shows that the numerically computed bifurcations do indeed appear to be converging to the leading-order components of \eqref{eq:CO3SN} and \eqref{eq:CO3PD}.
Notice the bifurcation values are asymptotically proportional to $\alpha^k$
(a slightly slower rate than that in Fig.~\ref{fig:B}).

\begin{figure}[b!]
\begin{center}
\setlength{\unitlength}{1cm}
\begin{picture}(14,5.1)
\put(0,0){\includegraphics[width=6.8cm]{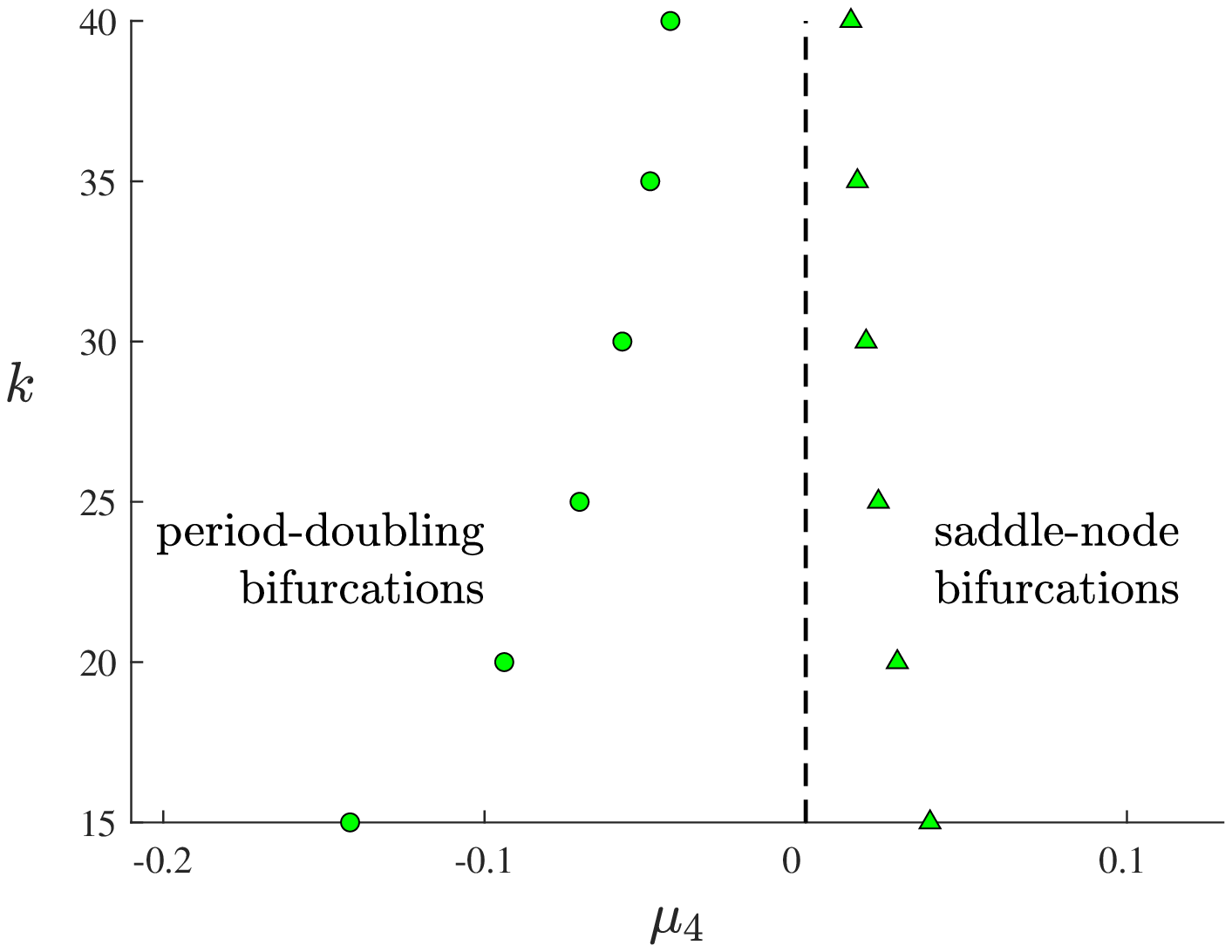}}
\put(7.2,0){\includegraphics[width=6.8cm]{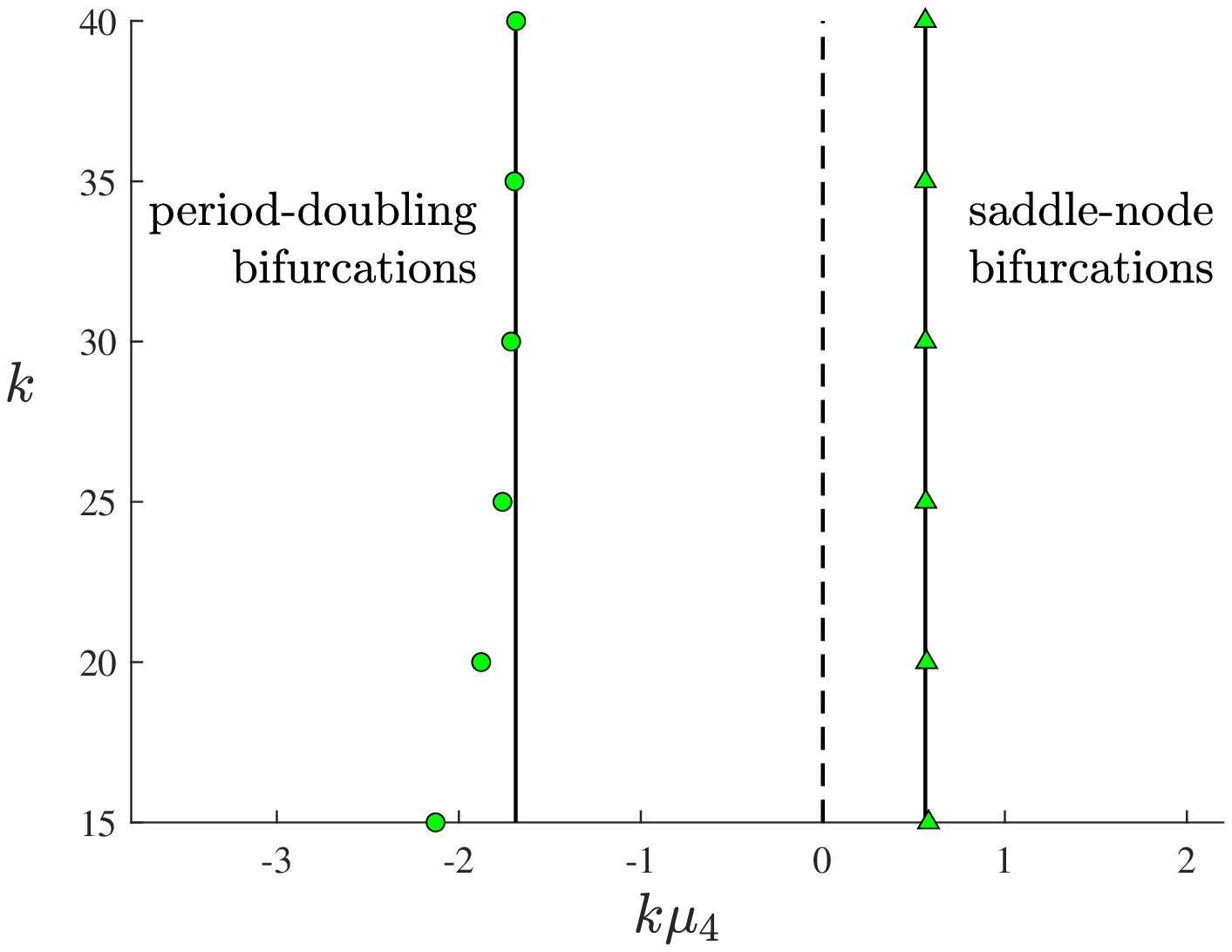}}
\put(0,5.1){\small \bf a)}
\put(7.2,5.1){\small \bf b)}
\end{picture}
\caption{Panel (a) is a numerically computed bifurcation diagram of \eqref{eq:toy} with \eqref{eq:toyParameters} and $\mu_1 = \mu_2 = \mu_3 = 0$.  Panel (b) shows convergence to the leading-order terms of \eqref{eq:CO4SN} and \eqref{eq:CO4PD}.
\label{fig:D}
} 
\end{center}
\end{figure}

Finally we fix $\mu_1 = \mu_2 = \mu_3 = 0$ and vary the value of $\mu_4$
which breaks the condition on the resonance terms in $T_0$.
As in the previous case Theorem \ref{th:main} does not apply.  By again calculating the bifurcations as above we obtain
\begin{equation}
    \epsilon_{{\rm SN}} = \frac{(1-c_{2,0})^2}{4d_{5,0}k} + \mathcal{O}\left(\frac{1}{k^2}\right),
    \label{eq:CO4SN}
\end{equation}
\begin{equation}
    \epsilon_{{\rm PD}} = -\frac{3(1-c_{2,0})^2}{4d_{5,0}k}+ \mathcal{O}\left(\frac{1}{k^2}\right),
    \label{eq:CO4PD}
\end{equation}
and these agree with the numerically computed bifurcation values as shown in Fig.~\ref{fig:D}. The bifurcation values are asymptotically proportional to $\frac{1}{k}$ which is substantially slower than in the previous three cases.






\section{Discussion}
\label{sec:conc}
\setcounter{equation}{0}

In this paper we have considered globally resonant homoclinic tangencies in smooth two-dimensional maps and determined scaling laws for the size of parameter intervals in which single-round periodic solutions are asymptotically stable.  We have illustrated the results with an abstract four-parameter family.  It remains to identify globally resonant homoclinic tangencies in prototypical maps and maps derived from physical applications.

About a parameter point satisfying the conditions of Theorem \ref{le:sufficientConditions}, for any positive $j \in \mathbb{Z}$ there exists an open region of parameter space in which the family of maps has $j$ coexisting asymptotically stable periodic solutions.  It follows from Theorem \ref{th:main} that the largest ball (sphere) that the region contains has a diameter asymptotically proportional to $|\alpha|^{2 j}$.  But, as we have shown, different directions of perturbation yield different scaling laws.  Consequently we expect such regions to have an elongated shape for large values of $j$.  Indeed preliminary investigations reveal that such regions may have a particularly complicated shape, bounded by many of the saddle-node and periodic-doubling bifurcations identified above.

We believe the primary $|\alpha|^{2 k}$  scaling law  holds true for  higher-dimensional maps. Certainly similar aspects of homoclinic tangencies have been shown to be independent of dimension, \cite{GoSh97}. Also in the piecewise-linear setting, globally resonant homoclinic tangencies were analysed in \cite{SiTu17}.

\appendix

\section{Proof of Lemma \ref{le:T0k}}
\label{app:proof1}
\setcounter{equation}{0}

Write
\begin{equation}
T_0(x,y) = \begin{bmatrix}
\lambda x \left( 1 + a_1 x y + x^2 y^2 \tilde{F}(x,y) \right) \\
\sigma y \left( 1 + b_1 x y + x^2 y^2 \tilde{G}(x,y) \right)
\end{bmatrix}.
\label{eq:T02}
\end{equation}
Let $R > 0$ be such that
\begin{equation}
|a_1|, |b_1|, |\tilde{F}(x,y)|, |\tilde{G}(x,y)| \le R
\label{eq:boundR}
\end{equation}
for all $(x,y) \in \cN$ and all sufficiently small values of $\mu$. For simplicity we assume $\alpha > 0$; if instead $\alpha < 0$ the proof can be completed in the same fashion.

We have $\lambda = \alpha + \cO \left( |\alpha|^k \right)$
and $\sigma = \frac{1}{\alpha} + \cO \left(|\alpha|^k \right)$.
Thus there exists $M \ge 2 R$ such that
$\lambda \le \alpha \left( 1 + M \alpha^k \right)$ and
$\sigma \le \frac{1}{\alpha} \left( 1 + M \alpha^k \right)$
for sufficiently large values of $k$.
It follows (by induction on $j$) that
\begin{equation}
\sigma^j \le \alpha^{-j} \left( 1 + 2 M j \alpha^k \right),
\label{eq:T0kproof11}
\end{equation}
and
\begin{equation}
\left( \lambda \sigma \right)^j \le 1 + 4 M j \alpha^k,
\label{eq:T0kproof12}
\end{equation}
for all $j = 1,2,\ldots,k$
again assuming $k$ is sufficiently large.

Let $\ee > 0$.
Assume
\begin{align}
|x-1| &\le \ee, &
\left| \alpha^{-k} y - 1 \right| &\le \ee,
\label{eq:T0kproof20}
\end{align}
for sufficiently large values of $k$.
We can assume $\ee < 1 - \frac{1}{\sqrt{2}}$ so then
\begin{equation}
\frac{\alpha^k}{2} \le |x y| \le 2 \alpha^k.
\label{eq:T0kproof21}
\end{equation}

Write
\begin{equation}
T_0^j(x,y) = \begin{bmatrix}
\lambda^j x \left( 1 + j a_1 x y + x^2 y^2 \tilde{F}_j(x,y) \right) \\
\sigma^j y \left( 1 + j b_1 x y + x^2 y^2 \tilde{G}_j(x,y) \right)
\end{bmatrix}.
\label{eq:T0j}
\end{equation}
Below we will use induction on $j$ to show that
\begin{equation}
\left| \tilde{F}_j(x,y) \right|, \left| \tilde{G}_j(x,y) \right| \le 12 M j^2,
\label{eq:T0kproof30}
\end{equation}
for all $j = 1,2,\ldots,k$, assuming $k$ is sufficiently large.
This will complete the proof because with $j=k$, \eqref{eq:T0kproof30} implies \eqref{eq:T0k}.

Clearly \eqref{eq:T0kproof30} is true for $j = 1$:
$\left| \tilde{F}_1(x,y) \right| = \left| \tilde{F}(x,y) \right| \le R \le \frac{M}{2} < 12 M$
and similarly for $\tilde{G}_1$.
Suppose \eqref{eq:T0kproof30} is true for some $j < k$.
It remains for us to verify \eqref{eq:T0kproof30} for $j+1$.
First observe that by using $|a_1| \le R$, \eqref{eq:T0kproof21}, and the induction hypothesis,
\begin{equation}
\left| 1 + j a_1 x y + x^2 y^2 \tilde{F}_j(x,y) \right|
\le 1 + 2 R j \alpha^k + 48 M j^2 \alpha{2 k}.
\nonumber
\end{equation}
For sufficiently large $k$ this implies
\begin{equation}
\left| 1 + j a_1 x y + x^2 y^2 \tilde{F}_j(x,y) \right|
\le 1 + 2 M j \alpha^k,
\label{eq:T0kproof40}
\end{equation}
where we have also used $M \ge 2 R$.
Similarly
\begin{equation}
\left| 1 + j b_1 x y + x^2 y^2 \tilde{G}_j(x,y) \right|
\le 1 + 2 M j \alpha^k.
\label{eq:T0kproof41}
\end{equation}

Write $T_0^j(x,y) = (x_j,y_j)$.
By using \eqref{eq:T0kproof11}, \eqref{eq:T0kproof20}, and \eqref{eq:T0kproof41} we obtain
\begin{equation}
|y_j| \le \alpha^{k-j} (1+\ee) \left( 1 + 2 M j \alpha^k \right)^2.
\nonumber
\end{equation}
Thus $|y_j| \le \alpha^{k-j} (1 + 2 \ee)$, say, for sufficiently large values of $k$.
Also $|x_j y_j|$ is clearly small, so we can conclude that $(x_j,y_j) \in \cN$
(in particular we have shown that $(x_{k-1},y_{k-1})$
can be made as close to $\left( 0, \frac{1}{\alpha} \right)$ as we like).

By matching the first components of
$T_0^{j+1}(x,y) = \left( T_0 \circ T_0^j \right)(x,y)$ we obtain
\begin{align}
\lambda^{j+1} x \left( 1 + (j+1) a_1 x y + x^2 y^2 \tilde{F}_{j+1}(x,y) \right) &=
\lambda^{j+1} x \left( 1 + j a_1 x y + x^2 y^2 \tilde{F}_j(x,y) \right) \nonumber \\
&\quad+ \lambda^{j+1} a_1 x^2 y (1 + P)
+ \lambda^{j+1} x^3 y^2 Q,
\label{eq:T0kproof50}
\end{align}
where
\begin{align}
1 + P &= \left( \lambda \sigma \right)^j 
\left( 1 + j a_1 x y + x^2 y^2 \tilde{F}_j(x,y) \right)^2
\left( 1 + j b_1 x y + x^2 y^2 \tilde{G}_j(x,y) \right),
\label{eq:T0kproof51} \\
Q &= \left( \lambda \sigma \right)^{2 j} 
\left( 1 + j a_1 x y + x^2 y^2 \tilde{F}_j(x,y) \right)^3
\left( 1 + j b_1 x y + x^2 y^2 \tilde{G}_j(x,y) \right)^2
\tilde{F}(x_j,y_j).
\label{eq:T0kproof52}
\end{align}
By \eqref{eq:T0kproof12}, \eqref{eq:T0kproof40}, and \eqref{eq:T0kproof41}, we obtain
\begin{align}
1 + P &\le \left( 1 + 4 M j \alpha^k \right) \left( 1 + 2 M j \alpha^k \right)^3 \nonumber \\
&\le 1 + 11 M j \alpha^k, \nonumber \\
Q &\le \left( 1 + 4 M j \alpha^k \right)^2 \left( 1 + 2 M j \alpha^k \right)^5 R \nonumber \\
&\le 2 R, \nonumber
\end{align}
assuming $k$ is sufficiently large and
where we have also used $|\tilde{F}(x_j,y_j)| \le R$ (valid because $(x_j,y_j) \in \cN$).
From \eqref{eq:T0kproof50},
\begin{equation}
\tilde{F}_{j+1}(x,y) = \tilde{F}_j(x,y) + \frac{P}{x y} + Q.
\nonumber
\end{equation}
Then by using the induction hypothesis,
the lower bound on $|x y|$ \eqref{eq:T0kproof21},
and the above bounds on $P$ and $Q$, we arrive at
\begin{align}
\left| \tilde{F}_{j+1}(x,y) \right| &\le 12 M j^2 + 22 M j + 2 R \nonumber \\
&\le 12 M j^2 + 24 M j \nonumber \\
&< 12 M (j+1)^2. \nonumber
\end{align}
In a similar fashion by matching the second components of
$T_0^{j+1}(x,y) = \left( T_0 \circ T_0^j \right)(x,y)$
we obtain $\left| \tilde{G}_{j+1}(x,y) \right| < 12 M (j+1)^2$.
This verifies \eqref{eq:T0kproof30} for $j+1$
and so completes the proof.
\hfill $\Box$

{\footnotesize
\bibliographystyle{plain}
\bibliography{UnfRef}
}

\end{document}